\documentclass[10pt,a4paper]{amsart}
\usepackage{amsfonts,amsmath,amssymb}
\usepackage{hyperref}
\usepackage[all]{xy}

\usepackage{amssymb,amsthm,amsxtra}
\usepackage[usenames]{color}

\usepackage{amscd}
\usepackage{amsthm}
\usepackage{amsfonts}
\usepackage{amssymb}
\usepackage{mathrsfs}
\usepackage{enumerate}

\newtheorem*{theorem*}{Theorem}
\newtheorem*{remark*}{Remark}
\newtheorem{lemma}{Lemma}[section]

\newtheorem{remark}[lemma]{Remark}

\newtheorem{theorem}[lemma]{Theorem}
\newtheorem{definition}[lemma]{Definition}
\newtheorem{notation}[lemma]{Notation}

\newtheorem*{conjecture*}{Conjecture}

\newtheorem{thm}[lemma]{Theorem}
\newtheorem{prop}[lemma]{Proposition}
\newtheorem{lem}[lemma]{Lemma}

\newtheorem{cor}[lemma]{Corollary}

\newtheorem{rem}[lemma]{Remark}

\newtheorem{introtheorem}{Theorem}

\baselineskip 18pt \textwidth 16cm  \theoremstyle{plain}

\newcommand{\tr}{\operatorname{Tr}}

\newcommand{\Hom}{\operatorname{Hom}}

\newcommand{\diag}{\operatorname{diag}}
\newcommand{\sgn}{\operatorname{sgn}}

\newcommand{\cc}{\mathbb{C}}

\newcommand{\rr}{\mathbb{R}}

\newcommand{\eps}{\varepsilon}

\newcommand{\Ind}{\operatorname{Ind}}

\newcommand{\re}{\operatorname{Re}}

\newcommand{\Z}{{\mathbb Z}}
\newcommand{\bZ}{{\mathbb Z}}

\newcommand{\R}{{\mathbb R}}
\newcommand{\bR}{{\mathbb R}}
\newcommand{\C}{{\mathbb C}}

\newcommand{\Span}{{\operatorname{Span}}}

\newcommand{\alp}{{\alpha}}

\newcommand{\Fre}{{Fr\'{e}chet \,}}

\newcommand{\Fou}{{\mathcal{F}}}

\newcommand{\Supp}{\mathrm{Supp}}

\newcommand{\GL}{\operatorname{GL}}
\newcommand{\Mat}{\operatorname{Mat}}
\newcommand{\oO}{\operatorname{O}}

\newcommand{\U}{\operatorname{U}}
\newcommand{\Sp}{\operatorname{Sp}}

\newcommand{\size}{\operatorname{size}}
\newcommand{\SL}{\operatorname{SL}}
\newcommand{\Id}{\operatorname{Id}}
\newcommand{\gl}{{\mathfrak{gl}}}

\newcommand{\Fr}{\operatorname{Fr}}

\newcommand{\Sym}{\operatorname{Sym}}

\newcommand{\Sc}{{\mathcal S}}

\newcommand{\sign}{\operatorname{sgn}}

\newcommand{\oQ}{\overline{Q}}

\newcommand{\fp}{\mathfrak{p}}
\newcommand{\fq}{\mathfrak{q}}
\newcommand{\fn}{\mathfrak{n}}

\newcommand{\cN}{\mathcal{N}}

\newcommand{\chH}{\widehat{H}}

\newcommand{\EitanA}[1]{{{#1}}}

\newcommand{\Dima}[1]{{{#1}}}
\newcommand{\DimaB}[1]{{{#1}}}
\newcommand{\DimaA}[1]{{{#1}}}
\newcommand{\DimaC}[1]{{{#1}}}
\newcommand{\DimaD}[1]{{{#1}}}
\newcommand{\dd}{\gamma}
\begin{document}

\author{Dmitry Gourevitch}
\address{Dmitry Gourevitch, Faculty of Mathematics and Computer Science, Weizmann
Institute of Science, POB 26, Rehovot 76100, Israel }
\email{dimagur@weizmann.ac.il}
\urladdr{http://www.wisdom.weizmann.ac.il/~dimagur}

 \author{Siddhartha Sahi}
 \address{Siddhartha Sahi,
 Department of Mathematics,
 Rutgers University,
 Hill Center - Busch Campus,
 110 Frelinghuysen Road
 Piscataway, NJ 08854-8019, USA}
 \email{sahi@math.rugers.edu}

 \author{Eitan Sayag}
 \address{Eitan Sayag,
 Department of Mathematics,
Ben Gurion University of the Negev,
P.O.B. 653,
Be'er Sheva 84105,
ISRAEL}
 \email{eitan.sayag@gmail.com}

\date{\today}
\title
{Invariant Functionals on \DimaA{Speh Representations}}
%
%
%
%
%
%
%
%
%
\begin{abstract}
\DimaA{We study $\Sp_{2n}(\R)$-invariant functionals on the spaces of smooth vectors in  Speh representations of $\GL_{2n}(\R).$

For even $n$ we give  expressions for such invariant functionals using an explicit realization of the space of smooth vectors in the
Speh representations.
Furthermore, we show that the functional we construct is, up to a constant, the unique  functional on the Speh representation which
is  invariant under the Siegel parabolic subgroup of  $\Sp_{2n}(\R)$.
For odd $n$ we show that the Speh representations do not admit an invariant functional with respect to the subgroup $\U_n$ of
$\Sp_{2n}(\R)$ consisting of unitary matrices.

Our construction, combined with the argument in \cite{GOSS}, gives a purely local and explicit construction of Klyachko models for
all unitary representations of $\GL_{n}(\R)$.}
\end{abstract}

\maketitle


\section{Introduction}

\DimaA{
In recent years, there has been considerable interest in periods of
automorphic forms in relation to the Langlands program and equidistribution
problems (\cite{SV,V}). The study of periods admits a local counterpart: the
study of invariant linear functionals and the concomitant
notion of \emph{distinction} of a representation $\pi$ of a reductive group
$G$ with respect to a subgroup $H\subset G$. We recall that a representation
$\pi$ is called \textbf{distinguished} with respect to a subgroup $H\subset G$
if the \textbf{multiplicity space}
$\Hom_{H}(\pi^{\infty},\cc)$ of $H$-invariant continuous functionals on the space $\pi^{\infty}$ of smooth vectors of $\pi$ is
non-zero.
In many interesting cases the pair $(G,H)$ is a Gelfand pair,
which means that the dimension of  the multiplicity space is at most one for any irreducible admissible representation $\pi$ of $G$.
This allows one to connect the global period integral to local linear
functionals.
Motivated by the work of Jacquet-Rallis \cite{JR} and Heumos-Rallis \cite{HR},
the third author together with O. Offen classified in
\cite{OSDist,OSConst,OSUn,OSsl2} those unitary representations of
$\operatorname{GL}_{2n}(F)$ that are distinguished with respect to the
subgroup $\operatorname{Sp}_{2n}(F)$, in the case that $F$ is a
non-archimedean local field. The case of archimedean $F$ was treated
subsequently in \cite{GOSS,AOS}. We remark that the pair $\operatorname{Sp}%
_{2n}(F)\subset\operatorname{GL}_{2n}(F)$ is a Gelfand pair (see
\cite{OSUn,AS,S}).

The classification of $\Sp_{2n}(\rr)$-distinguished unitary
representations of $\GL_{2n}(\rr)$ involves the family of
unitary representations discovered by B. Speh (\cite{Speh}). We recall that
these unitary representations and their generalizations to $\operatorname{GL}%
_{n}(F)$, where $F$ is a local field, play a central role in the Tadic-Vogan
classification of the unitary dual of $\operatorname{GL}_{n}(F)$. \Dima{To describe this classification} we use the
Bernstein-Zelevinsky notation $\pi_{1}\times\pi_{2}$ for (normalized)
parabolic induction from $\operatorname{GL}_{n_{1}}(F)\times\operatorname{GL}%
_{n_{2}}(F)$ to $\operatorname{GL}_{n_{1}+n_{2}}(F)$. For a discrete
series representation $\sigma$ of $\operatorname{GL}_{r}(F)$ we denote by
$U(\sigma,n)$ the corresponding Speh representation of $\operatorname{GL}%
_{nr}(F)$, and by
\[
\pi(\sigma,n,\alpha):=U(\sigma,n)|\cdot|^{\alpha}\times U(\sigma
,n)|\cdot|^{-\alpha},\text{ }0<\alpha<\Dima{\frac{1}{2}}
\]
the corresponding Speh complementary series representation.

Then any irreducible unitary representation of $\operatorname{GL}_{m}(F)$ can
be written in the form%
\begin{equation}
\pi=\pi_{1}\times\cdots\times\pi_{k},\label{=pi-prod}%
\end{equation}
where each $\pi_{i}$ is either a $U(\sigma_{i},n_{i})$ or a $\pi\left(
\sigma_{i},n_{i},\alpha_{i}\right)  $, and such an expression is unique up to
reordering of the $\pi_{i}$ (see \cite{Tad,Vog}). The answer to the
distinction is summarized in the next theorem, which in the archimedean case
is a combination of \cite[Theorem A]{GOSS} and \cite[Theorem 1.1]{AOS}.

\begin{theorem*}
If $\pi$ is an irreducible unitary representation of $\operatorname{GL}_{2n}(F)$  as in (\ref{=pi-prod}), then $\pi$ is
$\Sp_{2n}(F)$-distinguished iff all the $n_{i}$ are even.
\end{theorem*}

One of the key steps in the proof is to show that the generalized Speh
representations $U(\sigma,n)$ with even $n$ are distinguished by the
symplectic group. The proof of this result in \cite{OSDist} and \cite{GOSS} is
based on a global argument involving periods of residues of automorphic
Eisenstein series.

Recall that for archimedean $F$ we have $r\leq2$, and if $F={\mathbb{C}}$ then
$r=1$. If $r=1$ then $U(\sigma,n)$ is a character of $\operatorname{GL}%
_{n}(F),$ and $\pi(\sigma,n,\alpha)$ is a Stein complementary series
representation of $\operatorname{GL}_{2n}(F).$ We denote by $D_{m}$ the
discrete series representations of $\operatorname{GL}_{2}({\mathbb{R}})$ and
by $\delta_{m}$ the corresponding Speh representations of $\operatorname{GL}%
_{2n}({\mathbb{R}})$. In \cite{SaSt} the Speh representations $\delta_{m}$ of
$\operatorname{GL}_{2n}({\mathbb{R}})$ have been constructed explicitly as
natural Hilbert spaces of distributions on matrix space. The paper \cite{SaSt}
also describes and uses a construction of the Speh representations as
quotients of  degenerate principal series representations induced from
characters of the $(n,n)$ standard parabolic subgroup (see
\S \ref{subsec:SpehRep} below).
}

In the present paper we use the explicit constructions of \cite{SaSt} and give a direct proof that the spaces of
$\Sp_{2n}(\R)$-invariant functionals on the Speh representations of $\GL_{2n}(\rr)$ are zero if $n$ is odd and
one-dimensional if $n$ is even.
We also analyze functionals invariant with respect to subgroups of $\Sp_{2n}(\R).$

To describe our result we need some further notation.
Let $G:=G_{2n}$ denote the group $\GL_{2n}(\R)$. Let $\omega_{2n}$ be the
standard symplectic form on $\R^{2n}.$ More explicitly $\omega_{2n}$ is
given by $\begin{pmatrix}
  0 & \Id_n \\
  -\Id_n & 0
\end{pmatrix}$ and let $H:=H_{2n}=Sp_{2n}(\R)\subset G_{2n}$ denote the stabilizer of this form.
Let $$P:=\left \{  \begin{pmatrix}
  g & X \\
  0 & (g^{t})^{-1}
\end{pmatrix} \, \vert \, g \in \GL_n(\R), X\in \Mat_{n\times n}(\R), X=X^t
\right \} \subset H$$ denote the Siegel parabolic subgroup.
Let $\U_n\subset  H_{2n} \subset  G_{2n}$ be the unitary group.

In this paper we prove the following result.

\begin{introtheorem} \label{thm:main}
\begin{enumerate}[(i)]
\item \label{mainit:even} If $n$ is even then  $$\Hom_H(\delta_m^{\infty},\C)=\Hom_P(\delta_m^{\infty},\C)\simeq \C$$
\item \label{mainit:odd} If $n$ is odd then $$\Hom_H(\delta_m^{\infty},\C)=\Hom_{\U_n}(\delta_m^{\infty},\C)=\{0\}.$$
\end{enumerate}
\end{introtheorem}

It is known that the restriction of $\delta_{m}$ to $\SL_{2n}(\R)$ decomposes as a direct sum of two irreducible components
\DimaA{$\delta_{m}^{\pm}$}. It follows from Theorem \ref{thm:main} that exactly one of them admits an
$H$-invariant functional. In Lemma \ref{lem:SLNonVan} we determine \DimaA{that  $\delta_{m}^{+}$ does}.

It is easy to see that if $n$ is odd and $m$ is even then there
are no functionals on $\delta_m^{\infty}$ invariant with respect to $-\Id \in P\cap\U_n$, and thus neither $P$-invariant nor $\U_{n}$
-invariant
functionals exist (see Remark \ref{rem:noddmeven}).

\begin{remark*}
Although the pair $(G,P)$ is {\bf not} a Gelfand pair for
simple geometric reasons,  we show that the Speh
 representation $\delta_{m}$ still admits at most one $P$-invariant
 functional (at least for even $n$).
The reason we suspected this result to hold is that, as shown in  \cite{SaSt}, Speh representations stay irreducible when restricted to
a standard maximal parabolic subgroup $Q \subset G$ satisfying $Q\cap H = P$. It is
possible that $(Q,P)$ is a generalized Gelfand pair, {\it i.e.} the space of $P$-invariant functionals on the space of smooth vectors
of any irreducible unitary representation of $Q$ is at most one dimensional. However, this
statement \Dima{would still} not imply our uniqueness result, since the space of $G$-smooth vectors of $\delta_m$ could a priori
\EitanA{afford more continuous functionals.}
\end{remark*}

\subsection{Related results}

\EitanA{The present work was motivated by our previous results on Klyachko models for unitary representations of $\GL_{n}(\R).$} For
any $n$, any even $k \leq n$ and any field $F$, \cite{Kly} defines a subgroup $Kl_k$ of $\GL_n(F)$ and a generic character $\psi_k$ of
$Kl_k$. In particular,  $Kl_0$ is the group of upper unitriangular matrices and
$Kl_n= \Sp_n(F)$ (if $n$ is even). It is shown \Dima{in \cite{Kly,IS,HZ}  for finite fields $F$ and} in
\cite{HR,OSDist,OSConst,OSUn,OSsl2,GOSS,AOS} for local fields $F$
that for any irreducible unitary representation $\pi$ of $\GL_n(F)$ there exists a non-zero $(Kl_k,\psi_k)$-equivariant functional on
$\pi^{\infty}$ for exactly one $k$. The uniqueness of such functional is known only over
non-archimedean fields (see \cite{OSUn}).

The proof of existence of $k$ for $F=\R$, given in \cite{GOSS}, is \EitanA{achieved by reduction to the statement that certain
representations of $G=\GL_{2n}(\R)$ are $H=\Sp_{2n}(\R)$-distinguished}. This \EitanA{statement}
 is further reduced, using the Vogan classification of the unitary dual, to
\EitanA{an existence statement}
 of $H$-invariant functionals on the Speh representations (for even $n$).
\EitanA{Finally, the existence statement}
is proved using a global (adelic) argument. In
\EitanA{the present}
paper we give an explicit local construction of such a functional. Together with \cite{GOSS} this gives a proof of existence of
Klyachko models which uses only the representation theory of $\GL_n(\R)$ (and the theory of
distributions).

\EitanA{
The study  of invariant functionals  in this paper, and more broadly the study of multiplicity spaces belongs to the long and classical
tradition of branching laws (see e.g. \cite[Chapter 8]{GW}). In the context of symmetric
pairs and more generally in the context of spherical spaces, the basic result is that
these multiplicity spaces are finite dimensional (\cite{KO}, cf. \cite{KrSch}). Granted this qualitative result, one turns to the
question of precisely determining the dimension.
We note that in many interesting cases these spaces are at most one-dimensional (see e.g. \cite{vD,AG_HC,AGRS,SZ}). This multiplicity
one phenomenon has important consequences in number theory (\cite{Gross}).

In some situations  there are precise conjectures as to the dimensions of these multiplicity spaces (see e.g. \cite{GGP,Wal}) but in
general these dimensions are hard to determine, even in the context of symmetric pairs.
Another important task, motivated in part by the theory of automorphic forms,  is to construct a basis for these multiplicity spaces.
Recently, there has been a considerable interest in these aspects of the theory under the title of {\it symmetry breaking}.
The general theory of branching laws attempts the description of
symmetry breaking operators occurring in the general context of restrictions of representations
as in  \cite{KoSp14,KoSp}.
In particular, it is interesting to compare our main result with \cite[Chapter 14]{KoSp}.}

\EitanA{Another related result is the exact} branching of the representations $\delta_m^{\pm}$ of $\SL(4,\R)$ to $\Sp(4,\R)$ as
analyzed in \cite{OrSp}. It is shown there that the decomposition of \DimaA{$\delta_m^{-}$} is
discrete and multiplicity free, while the decomposition of \DimaA{$\delta_m^{+}$}  is continuous.
\subsection{Structure of the proof}
We use the realization of $\delta_m^{\infty}$ as the image of a certain intertwining differential operator
$\Box^m:\pi^{\Dima{\infty}}_{-m}\to \pi^{\Dima{\infty}}_m$, where $\pi^{\Dima{\infty}}_{-m}$ and
$\pi^{\Dima{\infty}}_m$ are  degenerate principal series representations induced from certain characters of a fixed $(n,n)$-parabolic
subgroup $\oQ\subset G$ (see \S  \ref{subsec:SpehRep}).

The study of the even case  is divided into two parts. In \S \ref{sec:Uni} we first use the
realization of $\delta_{m}^{\Dima{\infty}}$ as a quotient of the degenerate
principal series $\pi^{\Dima{\infty}}_{-m}$ to lift a linear $P$-invariant
functional on $\delta_{m}^{\Dima{\infty}}$ to an equivariant distribution on
$G$. More precisely, we study $P \times
\oQ$ equivariant distributions on $G_{}$. The technical heart is Corollary \ref{cor:Key}, which shows that such distributions
do not vanish on the open cell $N \oQ$. This is based on the techniques of \cite{AGS}, classical invariant theory and
a careful analysis of the double cosets $P\setminus G/\oQ$, which is postponed to \S \ref{subsec:PfKeyLem}.
\DimaA{Then we analyze the space of distributions  on the open
cell $N \oQ$ by identifying it with the space of distributions on $N$ with a certain equivariance property.} Identifying $N$ with its
Lie algebra and using the Fourier transform we show that this space is at most
one-dimensional for even $n$.
This finishes the proof of Proposition \ref{prop:FunUnique} which states that there exists at most one  $P$-invariant
 functional in the $n$ even case.

In the second part (\S \ref{sec:Funct}) we construct an $H$-invariant functional as an $H \times \oQ$-equivariant distribution on $G$.
For that we fix an explicit $H \times \oQ$-equivariant non-negative polynomial $p$,
consider the meromorphic family of distributions $p^{\lambda}$ (cf. \cite{Ber}) and take the principal part of this family at
$\lambda=(n-m)/2$\Dima{, {\it i.e.} the lowest non-zero coefficient in the Laurent expansion}. This
distribution defines an $H$-invariant functional on $\pi_m^{\infty}$. To show that the restriction of this functional to
$\delta_m^{\infty}$ is non-zero (Lemma \ref{lem:NonVan}) we use Corollary \ref{cor:Key} along with
another lemma from \S \ref{sec:Uni} 
on non-existence of equivariant distributions with certain support. The uniqueness of $P$-invariant functionals and the existence of
$H$-invariant ones imply that the two spaces are equal. Our proof shows that the spaces of
such functionals are equal and one-dimensional also for the (reducible) representations $\pi^{\Dima{\infty}}_m$ and
$\pi^{\Dima{\infty}}_{-m}$.

For odd $n$ we prove that already a $\U_n$-invariant functional does not exist (Corollary \ref{cor:Van}). We do that by analyzing the
$\oO_{2n}(\R)$-types of $\delta_m$  described in \cite{HL,Sah}
and showing that none of those have a $\U_n$-invariant
vector.

To summarize,  Theorem \ref{thm:main} follows from
Proposition \ref{prop:FunUnique} on uniqueness of $P$-invariant functionals for even $n$, Lemma \ref{lem:NonVan} on existence of
$H$-invariant functionals for even $n$ and Corollary \ref{cor:Van} on non-existence of $\U_{n}
$-invariant functionals for odd $n$.

\subsection{Acknowledgements}

The authors thank the Hausdorff Institute in Bonn for  perfect
working conditions during the summer of 2007 where the initial
collaboration on this project started.  They further thank Avraham Aizenbud, Joseph
Bernstein and Omer Offen for fruitful discussions on
the subject matter of this paper, \DimaD{and Itay Glazer for finding a typo in a previous version}.
\EitanA{We thank the anonymous referees for carefully reading the paper and making many valuable suggestions, in particular to mention
the
connection of our work to \cite{KO,KoSp14,KoSp,OrSp}}.

D.G. was partially supported by ISF grant 756/12 and a Minerva foundation grant.

E.S. was partially supported by ISF grant 1138/10.
\section{Preliminaries} \label{sec:Prel}

\subsection{Notation} \label{subsec:Not}
Recall the notation $G=G_{2n}=\GL_{2n}(\R)$, and $H=H_{2n}=\Sp_{2n}(\R)\subset G$.  Let
$$Q:=\left \{  \begin{pmatrix}
  a & c \\
  0 & d
\end{pmatrix}
\in G\right \} \quad \oQ:=\left \{  \begin{pmatrix}
  a & 0 \\
  b & d
\end{pmatrix}
\in G\right \} \quad
N:=\left \{  \begin{pmatrix}
  \Id_n & c \\
  0 & \Id_n
\end{pmatrix}
\in G\right \}
.$$
Recall that $P$ denotes $Q\cap H$ and let $$M:=\left \{  \begin{pmatrix}
  g & 0 \\
  0 & (g^{t})^{-1}
\end{pmatrix}
\right \} \quad \text{and} \quad U:=\left \{  \begin{pmatrix}
  \Id_n & B \\
  0 & \Id_n
\end{pmatrix} \, | \, B=B^t
\right \} $$
denote the Levi subgroup and the unipotent radical of $P$.

For $g \in \Mat_{i\times i}(\bR)$ we denote $|g|:=|\det(g)|$ and $\sgn(g):=\sign(\det(g))$.

For $q = \begin{pmatrix}
  A & 0 \\
  B & D
\end{pmatrix} \in \oQ$ we denote $\dd(q):=|A||D|^{-1}$ and $\eps(q):=\sgn(D)$. 

For any integer $m$ let $L_m$ denote the character of $\oQ$ given by $L_m:= \eps^{m+1}
\dd^{-(n+m)/2}$. Let $\pi_m^{\Dima{\infty}}$ denote the  (unnormalized) induced representation
$\Ind_{\oQ}^G(L_m)$\Dima{, with the topology of uniform convergence on  $G/\oQ$ together with all the derivatives}.
Considering $N$ as an open subset of $G/\oQ$, one can restrict smooth vectors of $\pi^{\Dima{\infty}}_m$  to $N$. This restriction is
an embedding since $N$ is an open subset of $G/\oQ$. We sometimes identify $N$ and its Lie
algebra $\fn$ with  $\Mat_{n \times n}(\R)$ 
\DimaA{by
$$ \begin{pmatrix}
  1 & X \\
  0 & 1
\end{pmatrix} \mapsto X\, \text{ and }\begin{pmatrix}
  0 & X \\
  0 & 0
\end{pmatrix}\mapsto X.$$

}
This enables us to define the Fourier transform on $\fn$.
Denote by $M_n^+$ (respectively $M_n^-$) the subset of  $\Mat_{n \times n}(\R)$ consisting of matrices with nonnegative (resp.
nonpositive) determinant.
For $f\in \pi_m^{\infty}$
we denote its restriction to $\fn$
by $f|_{\fn}$. We denote the space of
all smooth functions obtained in this way by $\pi_m^{\infty}|_{\fn}$.

\subsection{Sahi-Stein realization of the Speh representations} \label{subsec:SpehRep} 

For any $m\in \bZ_{\geq 0}$ define
$$\chH_m :=\{f \in \Sc^*(\fn) \, | \, \widehat{f} \in
L^2(\fn,| x |^{-m}dx)\} \text { and }\chH_m^{\pm} :=\{f \in \chH_m \, | \, \Supp \widehat{f} \subset M_n^{\pm}\},$$
where $\Sc^*(\fn)$ denotes the space of tempered distributions  $\fn$. The $\chH_m$ and $\chH_m^{\pm}$ are Hilbert spaces with the
scalar product $$\langle f,g\rangle =
\langle \widehat{f},\widehat{g}\rangle_{L^2(\fn,| x |^{-m}dx)}.$$

Define an action of $Q$ on $\chH_m$ by
$$\delta_m(q)f(x):= L_m(q)f(a^{-1}(c+xd)), \text{ for }q=
\begin{pmatrix}
  a & c \\
  0 & d
\end{pmatrix},$$
or equivalently on the Fourier transform side by
$$\widehat{\delta_m(q)f}(\xi)=\exp(2\pi i \tr(cd^{-1}\xi)) L_{-m}^{-1}(q)\widehat{f}(d^{-1}\xi a).$$

Summarizing the main results of \cite{SaSt} we obtain
\begin{thm}[\cite{SaSt}]\label{thm:SaSt}Let $m\in \bZ_{\geq 0}$. Then
\begin{enumerate}[(i)]
 \item The action of $Q$ extends to a unitary representation $\delta_m$ of $G$ on $\chH_m$.
\item $(G, \delta_m, \chH_m)$ is isomorphic to the Speh representation of $G$.
\item There exists  an epimorphism $\pi^{\Dima{\infty}}_{-m} \to \delta^{\Dima{\infty}}_m$ and an embedding $\delta^{\Dima{\infty}}_m
    \subset \pi^{\Dima{\infty}}_m$. The latter is defined  by the inclusion
    $\delta_m^{\infty} \subset \pi_m^{\infty}|_{\fn}$.
\item \label{it:SL} The restriction of $\delta_m$ to $\SL(2n,\R)$ is a direct sum of two irreducible representations $\delta_m^{\pm}$
    , realized on the subspaces $\chH_m^{\pm}$.
\end{enumerate}
\end{thm}
Consider the determinant as a polynomial on $\mathfrak{n}$ and
let $\Box $ denote the corresponding differential operator.

\begin{thm}\label{thm:Box}
The operator $\Box ^{m}$ defines a continuous \DimaC{$\mathrm{SL}(2n,\R)$}-equivariant map $\pi
_{-m}^{\infty }\rightarrow \pi _{m}^{\infty }$ with image $\delta
_{m}^{\infty }.$
\end{thm}

\begin{proof}
\DimaC{We will prove a stronger statement: the operator $\Box ^{m}$ defines a continuous $G$-equivariant map $\pi
_{-m}^{\infty }\rightarrow \mathrm{sgn}\pi _{m}^{\infty }$ with image $\delta
_{m}^{\infty },$ where $\mathrm{sgn}\pi _{m}^{\infty }$ denotes the twist of $\pi _{m}^{\infty }$ by the sign character of $G$.}

By \cite[Proposition 2.3]{KV} (see also \cite{Boe}), the operator $\Box ^{m}$ defines a continuous $G$-equivariant map $\pi
_{-m}^{\infty }\rightarrow \DimaC{\mathrm{sgn}}\pi _{m}^{\infty }$, which is non-zero by \cite{SaSt}.
By  \cite[Theorems 3.4.2-3.4.4]{HL} (see also \cite{Sah})  $\pi^{\Dima{\infty}}_{-m}$ has unique composition
series in the strong sense, meaning that any quotient of $\pi^{\Dima{\infty}}_{-m}$ has a unique irreducible subrepresentation, and all
these irreducible subquotients are pairwise non-isomorphic. It is easy to see that
$\pi^{\Dima{\infty}}_m$ is dual to $\pi^{\Dima{\infty}}_{-m}$ and thus their composition series are opposite. \DimaC{The composition
series are described in \cite{HL} in terms of their $K$-types, where $K=\oO(2n,\R)$ is the maximal compact subgroup of $G$,
and it is easy to see from this description that the set of $K$-types of the irreducible quotient of $\pi^{\Dima{\infty}}_{-m}$ is
invariant under multiplication by $\mathrm{sgn}$.
By the result of Casselman and Wallach (see \cite{CasGlob} or \cite[Chapter 11]{Wal2}), the category of smooth admissible \Fre
representations of moderate growth is abelian and any  morphism in it has
closed image.}
Hence the image of any nonzero intertwining operator from $\pi^{\Dima{\infty}}_{-m}$ to $\DimaC{\mathrm{sgn}}\pi^{\Dima{\infty}}_{m}$
is the unique irreducible \DimaC{quotient of $\pi^{\Dima{\infty}}_{-m}$}. Since $\delta_m^{\infty}$ is an irreducible  \DimaC{quotient
of $\pi^{\Dima{\infty}}_{-m}$,  the image of $\Box^m$ is
$\delta_m^{\infty}$.}
\end{proof}

\begin{remark}
One can deduce Theorem \ref{thm:Box} also from \cite{KS93}, which computes the action of $\Box^m$ on every $K$-type, where
$K=\oO(2n,\R)$. From the formula in \cite{KS93} and the description of the $K$-types of the
composition series of $\pi^{\Dima{\infty}}_{-m}$ in \cite{HL,Sah} one can see that $\Box^m$ does not vanish precisely on the $K$-types
of $\delta^{\Dima{\infty}}_m$.
\end{remark}

\subsection{Invariant distributions} \label{subsec:Dist}

\Dima{We will now recall some generalities on Schwartz functions and tempered distributions}.

\begin{definition} \label{def:Schwartz}
For an affine algebraic manifold $M$ we denote by $\Sc(M)$ the space of Schwartz functions on $M$, that is smooth functions $f$ such
that $df$ is bounded for any differential operator $d$ on $M$ with \Dima{algebraic}
coefficients.
We endow this space with a \Fre topology using the sequence of seminorms $\cN_d(f):=\sup_{x \in M}|df(x)|$, where $d$ is a
differential operator on $M$ with \Dima{algebraic} coefficients.
Also, for an algebraic vector bundle $E$ over $M$ we denote by $\Sc(M,E)$ the space of Schwartz sections of $E$.
We denote by $\Sc^*(M,E)$ the space of continuous linear functionals on $\Sc(M,E)$ and call its elements tempered distributional
sections. For a closed subvariety $Z\subset M$ we denote by $\Sc^*_M(Z,E)\subset \Sc^*(M,E)$ the
subspace of tempered distributional sections supported in $Z$.
For the theory of Schwartz functions and distributions on general semi-algebraic manifolds we refer the reader to \cite{AG_Sc}.
\end{definition}

\Dima{
\begin{notation}
\begin{itemize}
\item
For a manifold $M$ and closed submanifold $Z \subset M$ we denote by $N_Z^M:=TM|_Z/TZ$ the normal bundle to $Z$ in $M$ and by
$CN_Z^M\subset T^*M$ its dual bundle, {\it i.e.} the conormal bundle to $Z$ in $M$.

\item  For a point $z\in Z$ we denote by $N_{Z,z}^M$ the normal space at $z$ to $Z$ in $M$ and by $CN_{Z,z}^M$ the conormal space at
    $z$ to $Z$ in $M$.

\item
For a group $K$ acting on a vector space $V$ we denote by $V^K$ the subspace of $K$-invariant vectors and by $V^{K,\chi}$ the
subspace of vectors that change by the character $\chi$.
\item
If $K$ acts on a manifold $M$ we denote by $\Sc^*(M)^{K,\chi}$ the space of distributions on $M$ that change by the character $\chi$
under the action of $K$.

\item
For a real algebraic group $K$ we denote by $\Delta_K$ its modular character.
\end{itemize}
\end{notation}
}

\begin{theorem}[{\cite[\S B.2]{AGS}}] \label{thm:NashFilt}
Let a \Dima{real algebraic} group $K$ act on a real algebraic manifold $M$. Let $Z \subset M$ be a
\Dima{Zariski} closed subset. Let $Z = \bigcup_{i=1}^l Z_i$ be a
$K$-invariant stratification of $Z$. Let $\chi$ be a character of
$K$. Suppose that for any $k \in \Z_{\geq 0}$ and $1 \leq i \leq
l$, $$\Sc^*(Z_i,\Sym^k(CN_{Z_i}^M))^{K,\chi}=\{0\}.$$ Then
$\Sc^*_M(Z)^{K,\chi}=\{0\}.$
\end{theorem}

\begin{theorem}[Frobenius descent, see {\cite[Appendix B]{AG_HC}}] \label{thm:Frob}
Let a \Dima{real algebraic} group $K$ act on a \Dima{real algebraic} manifold $M$. Let $Z$
be a \Dima{real algebraic} manifold with a transitive action of $K$. Let
$\phi:M \to Z$ be a $K$-equivariant map.
Let $z \in Z$ be a point and $M_z:= \phi^{-1}(z)$ be its fiber.
Let $K_z$ be the stabilizer of $z$ in $K$.
Let $E$ be a $K$-equivariant \Dima{algebraic} vector bundle over $M$.

Then there exists a canonical isomorphism $$\Fr: (\Sc^*(M_z,E|_{M_z})
\otimes \Delta_K|_{K_z} \cdot \Delta_{K_z}^{-1})^{K_z} \cong
\Sc^*(M,E)^K.$$

\end{theorem}

From those two theorems we obtain the following corollary.
\begin{cor} \label{cor:BruFrob}
Let a \Dima{real algebraic} group $K$ act on a real algebraic manifold $M$. Let $Z \subset M$ be a
\Dima{Zariski} closed subset. Suppose that $Z$ has a finite number of orbits: $Z = \bigcup_{i=1}^l Kz_i$. Let $\chi$ be a character of
$K$. Suppose that for any  $1 \leq i \leq
l$ we have $$\Sym^*(N_{Kz_i,z_i}^M)^{K_{z_i},\chi \cdot \Delta_K|_{K_{z_i}} \cdot \Delta_{K_{z_i}}^{-1}}=\{0\},$$
where $\Sym^*$ denotes the
symmetric algebra.
Then
$\Sc^*_M(Z)^{K,\chi}=\{0\}.$
\end{cor}

\begin{lem} \label{lem:Integ}
 Let $K$ be a real algebraic group, and $R$ be a (closed) algebraic subgroup. Consider the right action of $R$ on $K$ and suppose that
 $K/R$ is compact. Let $\xi$ be a character of $R$. Then we have a natural isomorphism of
 left $K$ - representations
$$\DimaD{(\Ind_R^K(\xi))^* \cong} (C^{\infty}(K, \xi)^{R})^* \cong \Sc^*(K,\xi \DimaD{\Delta_K|_R} \Delta_R^{-1})^R \cong \Sc^*(K)^{(R,\DimaD{\xi \Delta_K|_R \Delta_R^{-1})}}.$$
\end{lem}
 \begin{proof}
\DimaD{The first and the last isomorphisms are straightforward. Let us prove the one in the middle.}

  Let $\mathfrak{Ind}(\xi)$ be the bundle on $K/R$ corresponding to $\xi$. Consider the surjective submersion $\pi:K \to K/R$. It
  defines an isomorphism $C^{\infty}(K, \xi)^{R} \cong C^{\infty}(K/R, \mathfrak{Ind}(\xi))$.

Since $K/R$ is compact, we have $C^{\infty}(K/R, \mathfrak{Ind}(\xi))^* \cong \Sc^*(K/R, \mathfrak{Ind}(\xi))$. Consider the diagonal
action of $K$ on $K \times K/R$ and the projections $p_1,p_2$ of $K \times K/R$ on both
coordinates. From Theorem \ref{thm:Frob} we obtain $$\Sc^*(K/R, \mathfrak{Ind}(\xi)) \cong \Sc^*(K \times K/R,p_1^*(\xi))^K \cong
\Sc^*(K,\xi \DimaD{\Delta_K|_R}\Delta_R^{-1})^R.$$

\end{proof}

\section{Uniqueness of $P$-invariant functionals} \label{sec:Uni}
\setcounter{lemma}{0}

In this section we assume that $n$ is even. The goal of this section is to prove the following proposition.
\begin{prop} \label{prop:FunUnique} For any integer $m$ we have
 $$\dim((\pi_{m}^\infty)^*)^P \leq 1.$$
\end{prop}

\DimaA{Recall the character  $L_m$ of $\oQ$ from \S \ref{subsec:Not} and note that $L_{-m}^{-1}=\eps^{m+1}\dd^{(n-m)/2}$}. Since
$\Delta_{\oQ} = \dd^{-n}$,  we obtain
from the definition of $\pi_{m}^{\Dima{\infty}}$ and Lemma \ref{lem:Integ}
\begin{equation}\label{=DistFunc}
(\pi_{m}^\infty)^*\cong \Sc^*(G)^{\overline{Q},\DimaA{L^{-1}_{-m}}}
\end{equation} and thus
in order to prove Proposition \ref{prop:FunUnique} we have to show that
for even $n$
$$\dim \Sc^*(G)^{P \times \overline{Q},1 \times \DimaA{L_{-m}^{-1}}} \leq 1.$$

We will need the following proposition, which we will prove in section \ref{subsec:PfKeyLem}.

\begin{prop}\label{prop:GeoKey}
Denote $K:=P\times \oQ$, and let $x \notin N\oQ$. Then
$$\Sym^*(
N_{Px\oQ,x}^{G}))^{K_x,\DimaA{L^{-1}_{-m}} \cdot \Delta_K|_{K_x} \Delta_{K_x}^{-1}} =\{0\}.$$
\end{prop}

From this proposition and Corollary \ref{cor:BruFrob} we obtain
\begin{cor}\label{cor:Key}
 $$\Sc^*_{G}(G - N \overline{Q})^{P \times \overline{Q},1 \times\DimaA{L^{-1}_{-m}}} = \{0\}.$$
\end{cor}

By this corollary it is enough to analyze $\Sc^*(N \overline{Q})^{P \times \overline{Q},1 \times\DimaA{L^{-1}_{-m}}}$.
Let $S$ denote the space of symmetric $n \times n$ matrices, and $A$ denote the space of anti-symmetric $n \times n$ matrices. Identify
$M \cong \GL_n(\bR)$ and let it act on $S$ and on $A$ by $x \mapsto g x g^t$.

\begin{lemma} \label{lem:Red2AntiSym}
We have
$$\Sc^*(N \overline{Q})^{P \times \overline{Q},1 \times\DimaA{L^{-1}_{-m}}} \cong  \Sc^*(A)^{\GL_n(\bR),\det^{1-m}} \cong
\Sc^*(A)^{\GL_n(\bR),\sgn^{m+1} | \cdot |^{m-n}} $$
\end{lemma}
\begin{proof}
Identify $U \cong S$ and let it act on itself by translations.
Then $N \oQ$ is isomorphic as a $P \times \oQ$-space to $A \times S \times \oQ$, where $\oQ$ acts on the third coordinate (by right
translations), $U$ acts on the second coordinate and $M$ acts on the first and the second
coordinates. Note that the action \Dima{ of $P \times \oQ$} on $S \times \oQ$ is transitive and that $\Delta_{\oQ} = \dd^{-n}$ and
$\Delta_P \begin{pmatrix}
g & 0\\
0 & (g^{-1})^t
\end{pmatrix} = |g|^{n+1}.$ The first isomorphism follows now from Frobenius descent.

The second isomorphism is given by Fourier transform on $A$ defined using the trace form.
\end{proof}

Let $O \subset A$ denote the open dense subset of non-degenerate
 matrices and $Z$ denote its complement.
The following lemma is a straightforward computation.
\begin{lemma} $\,$
 \begin{enumerate}[(i)]
 \item Every orbit \Dima{of $\GL_n(\bR)$} in $Z$ includes an element of the form $x := \begin{pmatrix}
     0_{k \times k} & 0
     \\
0 & \omega_{n-k}
     \end{pmatrix}$, \DimaA{for some even $k$}.
\item $N^{A}_{\GL_n(\bR)x,x} \cong \left \{ \begin{pmatrix}
    0_{k \times k} & b \\
0 & 0
     \end{pmatrix}\right \}\text{ and } \GL_n(\bR)_x = \left \{\begin{pmatrix}
     a_{k \times k} & 0 \\
c & d
     \end{pmatrix} \text{ such that }d\in \Sp_{(n-k)}  \right \}.$
\item $\Delta_{\GL_n(\bR)_x} = |\cdot|^{-(n-k)}. $
\end{enumerate}
\end{lemma}

\begin{cor}\label{cor:NoOddSym} 
\DimaA{For any $x\in Z$ we have }
 $$\Sym^*(N^{A}_{\GL_n(\bR)x,x})^{\GL_n(\bR)_x, \sgn^{m+1} | \cdot |^{m-n} \cdot \Delta_{\GL_n(\bR)_x}^{-1} } = \{0\}.$$
\end{cor}
\begin{proof}
From the previous lemma $\sgn^{m+1} | \cdot |^{m-n} \cdot \Delta_{\GL_n(\bR)_x}^{-1} =\sgn^{k+1} \det^{m-k}\DimaA{=\sgn \det^{m-k}}$.
This is not an algebraic character of $\GL_n(\bR)_x$ and thus there are no tensors that change under this character.
\end{proof}

\begin{cor}\label{cor:1dim}
 $$ \dim \Sc^*(A)^{\GL_n(\bR),\sgn^{m+1} | \cdot |^{m-n}} \leq 1.$$
\end{cor}
\begin{proof}
By Corollary \ref{cor:NoOddSym} and Corollary \ref{cor:BruFrob},
\begin{equation}\label{=NoDistInZ}
\Sc^*_A(Z)^{\GL_n(\bR),\sgn^{m+1} | \cdot |^{m-n}} = \{0\}.
\end{equation}
 Therefore, the restriction of equivariant distributions to $O$ is an embedding. Now, $$ \dim \Sc^*(O)^{\GL_n(\bR),\sgn^{m+1} | \cdot
 |^{m-n}} \leq 1,$$ since $O$ is a single orbit.
\end{proof}

Proposition \ref{prop:FunUnique} follows now from Corollary \ref{cor:1dim}, Lemma \ref{lem:Red2AntiSym}, Corollary \ref{cor:Key} and
\eqref{=DistFunc}.

\begin{remark}
\DimaA{ Corollary \ref{cor:Key}  does not extend to the case of odd $n$. For example, in this case the closed $P\times \oQ$-orbit $\oQ$
does support an equivariant distribution. }
\end{remark}

\section{Construction of the $H$-invariant functional} \label{sec:Funct}
\setcounter{lemma}{0}
Let $n$ be even. In this section we construct an $H$-invariant functional $\phi$ on $\pi_m^{\infty}$ for any $m\in \Z_{\geq 0}$ and
show that its restriction to $\delta_m^{\infty}$ is non-zero. Define a polynomial $p$ on
$\Mat_{2n\times 2n}(\R)$ by
\begin{equation}\label{=p}
p\begin{pmatrix}
  A & B \\
  C & D
\end{pmatrix} := \det(D^tB-B^tD)=\mathrm{Pfaffian}^2(D^tB-B^tD).
\end{equation}

Note that $p$ is non-negative, $H$-invariant on the left and changes under the right multiplication by $\oQ$ by the character $\vert
\cdot \vert \dd^{-1}$.
 Consider the meromorphic family of distributions on $\Mat_{2n\times 2n}(\R)$ given by\Dima{
$p^{\lambda}$. 
This family is defined for $\re \lambda>0$ and by \cite{Ber} has a meromorphic continuation (as a family of distributions) to the
entire complex plane.
For $\re \lambda>0$, the restriction of this distribution to \EitanA{$G=\GL_{2n}(\R)$ }
is a non-zero smooth function, and thus the restriction of the  family to $G$ is not identically zero. Define
\begin{equation} \eta^m_{\lambda}:=(p^{\lambda}|_G) |\cdot|^{-\lambda} \eps^{m+1}. \end{equation}
This is a tempered distribution, since $|\cdot|^{\lambda}$ is a smooth function on $G$ of moderate growth.
  Note that $$\eta^m_{\lambda}\in \Sc^*(G)^{(H\times \oQ,1 \times \eps^{m+1}\dd^\lambda)}.$$
Let $\alpha\in \Sc^*(G)$ be the principal part of this family at $\lambda=\frac{n-m}{2}$\Dima{, {\it i.e.} the lowest non-zero
coefficient in the Laurent expansion}.
By \eqref{=DistFunc} $\alp$ defines a non-zero $H$-invariant functional $\phi$ on $\pi_m^{\infty}$.}
$\pi_m^{\infty}$.


\begin{lem}\label{lem:NonVan}
$\phi|_{\delta_m^{\infty}}\neq 0$.
\end{lem}
\begin{proof}
By Theorem \ref{thm:Box} it is enough to show that $\Box^m\phi\neq 0$.
By Corollary \ref{cor:Key}, $\alpha|_{N\oQ}\neq 0$.  It is enough to show that $(\Box^m\alpha)|_{N\oQ}\neq 0$. As in \S \ref{sec:Uni},
let $A\subset N$ denote the subspace of anti-symmetric matrices and $O\subset A$ the open
subset of non-degenerate matrices. Note that $\alpha|_{N\oQ}\neq 0$ is $P\times \oQ$-equivariant and let
$\beta\in\Sc^*(A)^{\GL_n(\bR),\det^{1-m}}$ be the distribution on $A$ corresponding to $\alp$ by the Frobenius descent
(see Lemma \ref{lem:Red2AntiSym}). Note that $\Fou(\Box^m \beta)$ is $\Fou(\beta)$ multiplied by a polynomial. Thus it is enough to
show that $\Fou(\beta)$ has full support, {\it i.e.} $\Fou(\beta)|_O\neq 0$. This follows
from the equivariance properties of $\Fou(\beta)$ by \eqref{=NoDistInZ}. 
\end{proof}

This argument in fact proves slightly more.

\begin{lem}\label{lem:SLNonVan}
$\phi|_{(\delta_m^+)^{\infty}}\neq 0$.
\end{lem}
\begin{proof}
If $g$ is a Schwartz function on $M_{n}^{+}\subset N$ then its Fourier
transform $\widehat{g}$ defines a vector in $(\delta _{m}^{+})^{\infty }$
by Theorem \ref{thm:SaSt}. Thus it is enough to find such a $g$ for which $\zeta (\widehat{g}%
)\neq 0$, where $\zeta $ denotes the $P$-invariant distribution on $N$
corresponding to ${\alpha }$.

Let $f$ be a compactly supported smooth function on $O$ such that $\beta ({%
\mathcal{F}}(f))\neq 0$. Since the determinant is positive on $O$, there
exists a compact neighborhood $Z$ of zero in the space $S$ of symmetric $n$
by $n$ matrices such that $\mathrm{Supp}(f)+Z\subset M_{n}^{+}$. Let $h$ be
a smooth function on $S$ which is supported on $Z$ and s.t. $%
h(0)=1$. Let $g:=f\boxtimes h$ be the function on $N$ defined by $g(X+Y):=f(X)h(Y)$ where $X\in A$ and $Y\in S$. Let $\mathcal{F}_{S}$
denote the Fourier transform on $S$. Then we have
\begin{equation*}
\zeta (\widehat{g})=\zeta ({\mathcal{F}}(f)\boxtimes {\mathcal{F}}_{S}(h))=\beta ({\mathcal{F}}(f))\neq 0.
\end{equation*}%
\end{proof}

\begin{remark}
\begin{enumerate}[(i)]
\item For odd $n$, the polynomial $p$ is identically zero, since the matrix $D^tB-B^tD$ is an anti-symmetric matrix of size $n$.
\item The polynomial $p$ defines the open orbit of $H$ on $G/\oQ$. In general, one can show that if a linear complex algebraic group
    $\mathbf K$ acts \Dima{with finitely many orbits} on a complex affine algebraic manifold
    $\mathbf M$, both defined over $\R$,  $\mathbf W$ is a basic open subset of $M$ defined by a $\mathbf K$-equivariant polynomial
    $p$ with real coefficients, $\chi$ is a character of the group of  real points $K$ of
    $\mathbf K$ and there exists a non-zero $(K,\chi)$-equivariant  tempered distribution $\xi$ on $W$ then  there exists a non-zero
    $(K,\chi)$-equivariant  tempered distribution on $M$. Here, $W$ and $M$ denote the real
    points of $\mathbf W$ and $\mathbf M$.

To prove that
consider the analytic family of distributions $|p|^{\lambda}\xi$ on $W$. For $\re \lambda $ big enough, it can be extended to a
family $\eta_{\lambda}$ on $M$. By \cite{Ber} the family $\eta_{\lambda}$ has a meromorphic
continuation to the entire complex plane. Note that the distributions in this family are equivariant with a character that depends
analytically on $\lambda$. Thus taking the principal part at $\lambda=0$ we obtain a
non-zero $(K,\chi)$-equivariant  tempered distribution on $M$.

Note that since this construction involves taking principal part, the obtained distribution is not necessary an extension of the
original $\xi$.
This can already be seen in the case when $\mathbf M=\C$ is the affine line, $\mathbf W$ is the complement to 0 and $\mathbf K$ is
the multiplicative group $\C^{\times}$.
\end{enumerate}
\end{remark}

\section{Proof of Proposition \ref{prop:GeoKey}}\label{subsec:PfKeyLem}
We start from the description of the double cosets
$P\setminus G /\oQ$. Let $r_1,r_2,s,t$ be non-negative integers
such that $r_1+r_2+2s+2t=n$. We will view $2n \times 2n$ matrices
as $10 \times 10$ block matrices in the following way. First of
all, we view them as $2 \times 2$ block matrices with each block
of size $n \times n$. Now, we divide each block to $5 \times 5$
blocks of sizes $r_1,r_2,s,s,2t$ in correspondence. Denote by
$\sigma_{16}$ the permutation matrix that permutes blocks 1 and 6,
by $\sigma_{39}$ the permutation matrix that permutes blocks 3 and
9, and by $\tau_{5,10}$ the matrix which has $\begin{pmatrix}
  \Id_{2t} & \omega_{2t} \\
  0 & \Id_{2t}
\end{pmatrix}$
in blocks 5 and 10 and is equal to the identity matrix in other
blocks.
Recall the notation  $\omega_{2t}:=\begin{pmatrix}
  0 & \Id_t \\
  -\Id_t & 0
\end{pmatrix}.$
Denote \begin{equation}
x_{r_1,r_2,s,t}:=\sigma_{16} \sigma_{39}
\tau_{5,10}.
\end{equation}

\begin{lemma}\label{lem:cosets}
Each double coset in $P\setminus \GL_{2n}(\R)/\oQ$ includes a unique element
of the form $x_{r_1,r_2,s,t}$. The orbits in $N\oQ$ correspond to
$r_{\DimaB{1}}=s=0$.
\end{lemma}

\begin{proof}
\DimaB{Consider the Lagrangian subspaces $L:=\Span\{e_{1},\dots e_{n}\} \subset \R^{2n}$ and $L':=\Span\{e_{n+1},\dots e_{2n}\} \subset
\R^{2n}$. Note that $Q$ preserves $L$ and $\oQ$ preserves $L'$.
Identify $G/\oQ$ with the Grassmannian of $n$-dimensional subspaces of $\R^{2n}$ by $g\mapsto gL'$.}
 To an $n$-dimensional subspace $W\subset \R^{2n}$ we associate the following invariants: $$r_{\DimaB{1}}:=\dim L \cap W \cap W^\bot,\,
 r_\DimaB{2}:=\dim W^\bot \cap W - r_{\DimaB{1}}, \, s:=\dim L \cap W - r_{\DimaB{1}}, \,
 t:=(n-r_1-r_2)/2-s. $$

Note that $n-r_1-r_2$ is even since it is the rank of $\omega|_{W}$.
\DimaC{Note also that the identity $(L\cap W \cap W^{\bot})^{\bot}=L+W+W^{\bot}$ implies $n\geq r_1+r_2+2s$.}
\DimaB{Clearly, $W\in NL'$ if and only if $r_1=s=0$.}

Note the equality of vectors
\begin{equation}\label{=xW}
(v_1,0,v_2,0,\omega_{2t} u\, | \,0,w_2,w_1,0,u)^t = x_{r_1,r_2,s,t} (0,0,0,0,0 \, | \,v_1,w_2,w_1,v_2,u)^t.
\end{equation}
It is enough to show that $W$ can be transformed, using the action of $P$, to a space of vectors of the form
\DimaB{\eqref{=xW}}.

\DimaA{Let us first show that  $W$ can be transformed to a space of vectors of the form \begin{equation}\label{=Wform}
(v, Aw+Bv \, | \, Cw, w, Dw)^t,\quad \text{ where} \size(v) +\size(w) = n \text{ and }A \text{ is a square matrix.}
\end{equation}
 There exists a set $S$ of $n$ coordinates such that the projection of $W$ on the space of vectors that have zero coordinates from $S$
 is an isomorphism. Suppose that $S$ has $k$ of the coordinates $1\dots n$, and thus $n-k$
 of the coordinates $n+1,\dots 2n$. Note that acting by $M$ we can perform any permutation of the first $n$ coordinates followed  the
 same permutation on the last $n$ coordinates. Using such permutations we can transform $S$
 to the set \DimaB{$\{n-k+1,\dots,n,n+1,\dots n+l, n+k+l+1,\dots,2n\}$ for some $l\leq n-k$}. Then $W$ will have the form
 \eqref{=Wform}.
}

\DimaB{Let us now rewrite \eqref{=Wform} in more detailed form, using four blocks of the same sizes $y_i$ in the first $n$ coordinates
and last $n$ coordinates:}
$$ (v_1, \,  v_2, \, A_{11}w_1+ A_{12} w_2 + B_{11} v_1 + B_{12}v_2,  \, A_{21}w_1+ A_{22} w_2 + B_{21} v_1 + B_{22}v_2 \, | C_1 w_1 +
C_2 w_2, \, w_1, \, w_2, \, D_1w_1 + D_2 w_2)^t $$
Denote the first four blocks by $e_i$ and the last by $f_i$. For any \DimaB{ $i,j \in \{1,2,3,4\}$ with $i \neq j$}, $M=\GL_{n}(\R)$
allows us to do the following operations:
$$(1)_i \quad e_i \mapsto ge_i, \quad f_i \mapsto (g^{t})^{-1}f_i, \quad \DimaB{\text{where }g\in\GL_{y_i}(\R)}$$
$$(2)_{ij} \quad e_i \mapsto e_i+ae_j, \quad \DimaA{f_j} \mapsto f_j - \DimaA{a}^tf_i, \quad \DimaB{\text{where }a\in\Mat_{y_i\times
y_j}(\R)}.$$
Similarly, $U$ allows us to do two more operations:
$$(3)_{ij} \quad e_i \mapsto e_i+bf_j, \quad e_j \mapsto e_j + b^tf_i, \quad \DimaB{\text{where }b\in\Mat_{y_i\times y_j}(\R)}$$
$$(4)_i \quad e_i \mapsto e_i + (c+c^t)f_i, \quad \DimaB{\text{where }c\in\Mat_{y_i\times y_i}(\R).}$$
Using $(2)_{31}$ and $(2)_{41}$, and redefining $C$ and $D$ we get $B=0$.
Using $(2)_{21}$ and \DimaA{$(2)_{34}$}, and redefining $A$ we get $C=0$ and $D=0$.

Using $(3)_{32}$ and $(3)_{42}$ and $(3)_{43}$ we can arrange $A_{11}=A_{21}=A_{22} = 0$.
Using $(3)_{33}$ we make $A_{12}$ anti-symmetric. Now, using $(1)_3$ we can replace $A_{12}$ by $g A_{12} g^t$ and thus we can bring it
to the form $A_{12} = \begin{pmatrix}
                                         0 & 0\\
                                         0 & \omega_{2t}
                                         \end{pmatrix}.$

\end{proof}

\begin{lemma}[See \S \ref{subsubsec:PfInvDesc} below] \label{lem:InvDesc}
Let $K:=P \times \oQ$ and $x:=x_{r_1,r_2,s,t}$. Then

(i) If $s>0$ then $$\Sym^*(
N_{Px\oQ,x}^{G}))^{K_x,\DimaA{L_{-m}^{-1}}\cdot \Delta_K|_{K_x} \Delta_{K_x}^{-1}} =\{0\}.$$
(ii) If $s=0$ then
%
$$\Sym^*(
N_{Px\oQ,x}^{G}))^{K_x,\DimaA{L_{-m}^{-1}}\cdot \Delta_K|_{K_x} \Delta_{K_x}^{-1}}\cong
\Sym^*(\gl_{r_1})^{\GL_{r_1},|\cdot|^{-m-r_1}\sgn^{m+1}} \otimes
\Sym^*(o_{r_2})^{\GL_{r_2},\det^{2t-m+1}}$$

%
where $o_{r_2}$ denotes the space of
antisymmetric matrices and $\GL_{r_1}$ and $\GL_{r_2}$ act by $a \mapsto g a g^t$.
\end{lemma}

\begin{lemma} \label{lem:InvTheo}
Let $k,l \in \bZ_{\geq 0}, \,\, r \in \bZ_{>0}$.\\
(i) If $k \neq l \quad (mod\,\, 2)$ 
then
$$\Sym^*(\gl_{r})^{\GL_{r},|\cdot|^{k}\sgn^{l}}=\{0\}.$$
(ii) If $k\DimaB{\neq}0$ and $r$ is odd then
$$\Sym^*(o_{r})^{\GL_{r},\det^{k}}=\{0\}.$$
\end{lemma}
\begin{proof}$ $\\
(i)The only algebraic characters of $\GL_{r}$ are powers of the determinant. \\
(ii) The stabilizer in $\GL_r$ of every matrix in  $o_{r}$ has an element with determinant \DimaB{different from} 1.
\end{proof}

\begin{proof}[Proof of Proposition \ref{prop:GeoKey}]
\DimaB{By Lemma \ref{lem:cosets} it is enough to show that for  $x=x_{r_1,r_2,s,t}$ with $r_1+s>0$ we have $$\Sym^*(
N_{Px\oQ,x}^{G}))^{K_x,\DimaA{L_{-m}^{-1}}\cdot \Delta_K|_{K_x} \Delta_{K_x}^{-1}} =\{0\}.$$
If $s>0$ this follows from Lemma \ref{lem:InvDesc}(i).  Otherwise $r_1>0$ and, by Lemma \ref{lem:InvDesc}(i), we have to show that
\begin{equation}\label{=InvVan}
\Sym^*(\gl_{r_1})^{\GL_{r_1},|\det|^{-m-r_1}\sign(\det)^{m+1}} \otimes
\Sym^*(o_{r_2})^{\GL_{r_2},\det^{2t-m+1}}=\{0\}
\end{equation}
Note that since $n$ is even, $r_1$ and $r_2$ are of the same parity. If they are even then \eqref{=InvVan} follows from Lemma
\ref{lem:InvTheo}(i), and otherwise from Lemma \ref{lem:InvTheo}(ii).}
\end{proof}

\subsection{Proof of Lemma \ref{lem:InvDesc}} \label{subsubsec:PfInvDesc}

Let $x=x_{r_1,r_2,s,t}$ be as in the lemma. We need to compute the space $N_{x,Px\oQ}^{G}$, the stabilizer $K_x$ and its modular
function. In order to do that we compute the conjugates of $P$ and its Lie algebra $\fp$ under
$x$.

\begin{lem} \label{qConj}
Let $q := \begin{pmatrix}
  a & b \\
  0 & d
\end{pmatrix} \in \fq$. Then $x^{-1}qx = \begin{pmatrix}
  A & B \\
  C & D
\end{pmatrix},$ where\\

$$A= \begin{pmatrix}
         d_{11}&   0&         d_{14}&   0&   0  \\
         b_{21}& a_{22}&         b_{24}& a_{24}& a_{25} \\
         d_{41}&   0&         d_{44}&   0&   0 \\
         b_{41}& a_{42}&         b_{44}& a_{44}& a_{45} \\
 b_{51} - \omega_{\Dima{2t}} d_{51}& a_{52}& b_{54} - \omega_{\Dima{2t}} d_{54}& a_{54}& a_{55}
\end{pmatrix}$$
$\,$\\
$$B = \begin{pmatrix}
  0&         d_{12}&         d_{13}&   0&                 d_{15}\\
 a_{21}&         b_{22}&         b_{23}& a_{23}&         b_{25} + a_{25}\omega_{\Dima{2t}}\\
   0&         d_{42}&         d_{43}&   0&                 d_{45}\\
 a_{41}&         b_{42}&         b_{43}& a_{43}&         b_{45} + a_{45}\omega_{\Dima{2t}}\\
 a_{51}& b_{52} - \omega_{\Dima{2t}} d_{52}& b_{53} - \omega_{\Dima{2t}} d_{53}& a_{53}& b_{55} + a_{55}\omega_{\Dima{2t}} -
 \omega_{\Dima{2t}} d_{55}
\end{pmatrix}
$$
$ $\\
\quad
$$
C = \begin{pmatrix}
         b_{11}& a_{12}&         b_{14}& a_{14}& a_{15}\\
         d_{21}&   0&         d_{24}&   0&   0\\
         d_{31}&   0&         d_{34}&   0&   0\\
         b_{31}& a_{32}&         b_{34}& a_{34}& a_{35}\\
         d_{51}&   0&         d_{54}&   0&   0
\end{pmatrix}
\quad
D = \begin{pmatrix}
 a_{11} & b_{12}&         b_{13}& a_{13}&         b_{15} + a_{15}\omega_{\Dima{2t}}\\
   0&         d_{22}&         d_{23}&   0&                 d_{25}\\
  0&         d_{32}&         d_{33}&   0&                 d_{35}\\
 a_{31}&         b_{32}&         b_{33}& a_{33}&         b_{35} + a_{35}\omega_{\Dima{2t}}\\
 0&         d_{52}&         d_{53}&   0&                 d_{55}
\end{pmatrix}.$$
\end{lem}
This lemma is a straightforward computation.

We can identify $T_xG \cong \gl_{2n}$. Under this identification $T_xPx\oQ \cong x^{-1}\fp x  + \overline{\fq}$ and
$$N_{x,Px\oQ}^{G} \cong \gl_{2n} / (x^{-1}\fp x  + \overline{\fq}) \cong \fn/(\fn \cap (x^{-1}\fp x  + \overline{\fq})).$$

From the previous lemma we obtain
\begin{cor} \label{cor:norspace}
\DimaB{Recall the identification $n\cong \Mat_{n\times n}(\R)$ and}
let $V \subset \fn$ denote the subspace consisting of matrices of the form

$$ \begin{pmatrix}
    n_{11} & n_{12} & 0 & n_{14} & n_{15}\\
    n_{12}^t & n_{22} & 0  &  0     & 0\\
    n_{31} & 0 & 0  &  n_{34}    & 0\\
    0 & 0 & 0  &  0     & 0\\
    n_{15}^t & 0 & 0  &  0     & 0
   \end{pmatrix}
 ,$$
such that $n_{22}=-n_{22}^t.$

Then $V$ projects isomorphically onto  $\fn/(\fn \cap (x^{-1}\fp x  + \overline{\fq})).$
\end{cor}

Now let us analyze the stabilizer $K_x$.
From Lemma \ref{qConj} we obtain
\begin{cor} \label{cor:Kx}
$ $
\begin{enumerate}[(i)]
 \item \Dima{The Lie algebra $\fp \cap x \overline{\fq} x^{-1}$ consists  of matrices $ \begin{pmatrix}
  A & B \\
  0 & -A^{t}
\end{pmatrix}$  such that }
$$A = \begin{pmatrix}
  A_{11} & A_{12} & A_{13}& A_{14}& A_{15}\\
  0 & A_{22} & 0 & 0 & 0\\
  0 & A_{32} & A_{33} & 0 & \DimaA{-\omega_{2t}B_{35}}\\
  0 & A_{42} &  0 & A_{44}& \DimaA{\omega_{2t}B_{45}}\\
  0 & A_{52} &  0 & 0 & A_{55}
\end{pmatrix}, \quad B= \begin{pmatrix}
  B_{11} & B_{12} & B_{13}& B_{14}& B_{15}\\
  B_{12}^t & 0 & 0 & 0 & 0\\
  B_{13}^t & 0 & B_{33} & 0 & B_{35}\\
  B_{14}^t & 0 &  0 & B_{44}& B_{45}\\
  B_{15}^t & 0 & B_{35}^t & B_{45}^t & 0
\end{pmatrix},$$
$$A_{55}\in \mathfrak{sp}(2t), \,\, B_{11}=B_{11}^t,\,\, B_{33}=B_{33}^t,\,\, B_{44}=B_{44}^t.$$\ 

\item \Dima{Let $p= \begin{pmatrix}
  A & B \\
  0 & (A^{t})^{-1}
\end{pmatrix}\in P$. Let $k = (p , x^{-1}px) \in K_x$. The modular function of $K_x$ is given by }$$\Delta_{K_x}(k) =
|A_{11}|^{2n-r_1+1} |A_{22}|^{-n+r_1+r_2} |A_{33}|^{n-r_1-s+1}|A_{44}|^{n-r_1-s+1} .$$

\item Let $q  =  \begin{pmatrix}
  A & 0 \\
  C & D
\end{pmatrix} \in \oQ \cap x^{-1}Px$. Let $k = (x q x^{-1}, q) \in K_x$.
Then $k$ acts on $V$ by $$k \cdot n = pr_V(AnD^{-1}),$$
where $pr_V: \fn \to V$ denotes the projection.
\end{enumerate}
\end{cor}

\begin{cor}\label{cor:tau}
Denote $$\chi :=\DimaA{L^{-1}_{-m}}\cdot \Delta_K|_{K_x} \Delta_{K_x}^{-1}= \eps^{m+1}\dd^{(n-m)/2} \cdot \Delta_K|_{K_x}
\Delta_{K_x}^{-1}.$$
Let $$q = \diag(a,b,c,(c^{t})^{-1},\Id,(a^t)^{-1},(b^{t})^{-1},d,(d^{t})^{-1},\Id).$$ Let $k:=(x q x^{-1}, q) \in K_x$.
Then $$\chi(k) = (\sgn(a)\sgn(b)\sgn(c)\sgn(d))^{m+1} |a|^{-m-r_1} |b|^{2s+2t-m+1}|c|^{-r_1-s}|d|^{-r_1-s}.$$
\end{cor}
\begin{proof}
$$\dd(q) = |a|^2|b|^2 \quad \text{ and } \quad \Delta_{\oQ}(q) = |a|^{-2n}|b|^{-2n}$$

$$ x q x^{-1} = \diag((a^{t})^{-1}, b, (d^{t})^{-1}, (c^{t})^{-1},\Id,a, (b^{t})^{-1},d,c,\Id)$$

 $$\Delta_K(k) = |a|^{-3n-1} |b|^{-n+1} |c|^{-n-1} |d|^{-n-1}  $$

$$\Delta_{K_x}(k) = |a|^{-2n+r_1-1}|b|^{-n+r_1+r_2}|c|^{-n+r_1+s-1} |d|^{-n+r_1+s-1}$$
\end{proof}

Now we are ready to prove Lemma \ref{lem:InvDesc}.
\begin{proof}[Proof of Lemma \ref{lem:InvDesc}]
If $s>0$ then $\Sym^*(V)^{K_x, \chi}=0,$ since tensors cannot have negative homogeneity degrees. Otherwise, $V$ involves only $3$
blocks - the ones numbered $1,2$ and $5$.

Let $p \in
\Sym^*(V)^{K_x, \chi}$. Identify $K_x$ with  \Dima{$x^{-1}Px\cap\oQ$ using the projection on} the second coordinate.

Consider the action of the block $A_{21}$. It can map any non-zero vector in the block $n_{11}$ to any vector in the block $n_{12}$.
This action does not change any element in any other block of $V$ (it does effect $n_{22}$,
but not its anti-symmetric part). Also, the character $\chi$ does not depend on $A_{21}$. Therefore $p$ does not depend on the
variables in the block $n_{12}$.

In the same way, using the action of $A_{52}$, we can show that  $p$ does not depend on the variables in the block $n_{15}$. Therefore,
$p$ depends only on $n_{11}$ and $n_{22}$. Hence
$$\Sym^*(V)^{K_x, \chi} \cong
\Sym^*(\gl_{r_1})^{\GL_{r_1},|\cdot|^{-m-r_1}\sgn^{m+1}} \otimes
\Sym^*(o_{r_2})^{\GL_{r_2},|\cdot|^{2t-m+1}\sgn^{m+1}}.$$
\end{proof}

\section{Non-existence of an $H$-invariant functional for odd $n$} \label{sec:Non-exist}
\setcounter{lemma}{0}
In this section we prove that if $n$ is odd then there are no $\U_n$-invariant functionals on the Speh representations and therefore
there are no $H$-invariant functionals. We do that using $K$-type analysis. The maximal
compact subgroup of $G$ is $K:=\oO_{2n}(\R)$, and $\U_n=K\cap H$ is a symmetric subgroup of $K$. We show that no $K$-type of $\delta_m$
has a $\U _{n}$-invariant vector.

 The root system of $K$ is of type $D_{n}$,
and we make the usual choice of positive roots%
\begin{equation*}
\left\{ \varepsilon _{i}\pm \varepsilon _{j}:i<j\right\}
\end{equation*}%
where $\varepsilon _{i}$ is the $i$-th unit vector in $\mathbb{R}^{n}$. With
this choice, the highest weights of $K$-modules are given by
integer sequences $\mu =\left( \mu _{1},\ldots ,\mu _{n}\right) \in \mathbb{Z%
}^{n}$ such that%
\begin{equation}
\mu _{1}\geq \cdots \geq \mu _{n-1}\geq  \mu _{n}\geq 0 .
\label{=ai}
\end{equation}
\DimaC{If $\mu_n=0$ then two $K$-types correspond to the sequence $\mu$, they differ by the determinant character.}

\begin{remark}\label{rem:noddmeven}
From the definition of $\pi^{\Dima{\infty}}_m$ we see that if $n$ is odd and $m$ is even then
the central element $-\Id\in G$ acts
by  scalar $-1$, and  there are neither $P$-invariant nor $\U_{n}$ -invariant
functionals on $\delta_{m}^{\infty}$.
\end{remark}

\DimaC{Let $m\geq0$.} Since $\delta_m^{\infty}$ is the irreducible quotient of $\pi^{\Dima{\infty}}_{-m}$, the following theorem follows from \cite[Theorems
3.4.2 - 3.4.4]{HL} (see also \cite{Sah}).
\begin{thm}\label{lem:Ktypes}
The $K$-types of $\pi^{\Dima{\infty}}_{\pm m}$ are given by
sequences as in (\ref{=ai}) with $\mu_i \equiv  m+1 \, (\mathrm{mod}\,\, 2)$,
while the $K$-types of the Speh representation $\delta _{m}$ satisfy the
additional condition
$\mu _{n} \geq m+1$.
\end{thm}

\DimaC{
\begin{lem}\label{lem:noUtype}
If $n$ is odd then no $K$-type $(\mu_1,\dots,\mu_n)$ with $\mu_n\neq 0$ has $\U_{n}$-invariant vectors.

\end{lem}
\begin{proof}
Denote $K^0:=\mathrm{SO}_{2n}(\R)$. By \cite[Proposition 5.17]{Vog} any $K$-type with $\mu_n\neq 0$ decomposes into two $K^0$-types, one given by $(\mu_1,\dots,\mu_n)$
and the other by $(\mu_1,\dots,-\mu_n)$. Let us show that neither of them has $\U_{n}$-invariant vectors.

Let $\mathfrak{a}$ be a maximal Cartan subspace of the symmetric pair $(\mathfrak{k},\mathfrak{u}_n)$ and let $\mathfrak{t}$ be a maximal torus in the centralizer of $\mathfrak{a}$ in $\mathfrak{u}_n$.  By the Cartan-Helgason theorem (see \cite[Chapter V, Corollary 4.2]{HelGGA}) it suffices to show that the highest weights of the $K^0$-types mentioned above are not trivial when restricted to $\mathfrak{t}$. Examining the Satake diagram in \cite[Chapter X, \S F, Table VI, Case D III, odd r]{HelSym} we see that the Killing form identifies $\mathfrak{t}$ with the  span of $\{\eps_1-\eps_2,\eps_3-\eps_4,\dots,\eps_{n-2}-\eps_{n-1},\eps_n\}$. Thus the $K^0$-types that have $\U_{n}$-invariant vectors are of the form
$\mu_{2i-1}=\mu_{2i}$ for $1 \leq i \leq n/2$ and
$\mu_n=0$.
\end{proof}
}

\begin{cor}\label{cor:Van}
 If $n$ is  odd then there are no $\U_{n}$-invariant functionals on $\delta_m^{\infty}$.
\end{cor}
\begin{proof}
By Remark \ref{rem:noddmeven} we can
assume that $m$ is odd. Then by Lemma \ref{lem:noUtype} and Theorem \ref{lem:Ktypes}, no $K$-type of $\delta_m$ has a $\U_n$-invariant
vector. Therefore, the space of $K$-finite vectors, which decomposes to a direct sum of
$K$-types, does not have a $\U_{n}$-invariant functional. This space is dense in $\delta_m^{\infty}$, hence  there are no
$\U_{n}$-invariant functionals on $\delta_m^{\infty}$ either.
\end{proof}


\DimaC{
\begin{rem}
The pair $(K^0,\mathrm{U}_n)$, being a compact connected symmetric pair is a Gelfand pair. A similar argument to the one in the proof of Lemma \ref{lem:noUtype} shows that  $(K,\mathrm{U}_n)$ is also a Gelfand pair.
\end{rem}
}

\end{document}